\documentclass[11pt]{amsart}
\usepackage{amssymb}
\usepackage{amsmath}
\usepackage{amsthm}

\DeclareMathOperator{\Rep}{Re}
\DeclareMathOperator{\Imp}{Im}

\DeclareMathOperator{\Log}{Log}

\DeclareMathOperator{\li}{li}
\DeclareMathOperator{\mult}{mult}

\numberwithin{equation}{section}
\allowdisplaybreaks[4]

\theoremstyle{plain}
\newtheorem{Lemma}{Lemma}[section]
\newtheorem{Theorem}{Theorem}[section]
\newtheorem{Proposition}[Lemma]{Proposition}
\newtheorem{Corollary}[Lemma]{Corollary}

\theoremstyle{definition}

\newtheorem*{Remark}{Remark}

\newcommand{\R}{\mathbb {R}}
\newcommand{\C}{\mathbb {C}}
\newcommand{\Z}{\mathbb {Z}}

\begin{document}
\title[Zeros of $L'(s,\chi)$]
{
Zeros of the first derivative of Dirichlet $L$-functions
}
\author[H. Akatsuka]{Hirotaka Akatsuka}
\address{%
Otaru University of Commerce,
3--5--21, Midori, Otaru, Hokkaido, 047--8501,
Japan.
}
\email{akatsuka@res.otaru-uc.ac.jp}
\author[A. I. Suriajaya]{Ade Irma Suriajaya}
\address{%
Nagoya University,
Furo-cho, Chikusa-ku, Nagoya, Aichi, 464--8602,
Japan.
}
\email{m12026a@math.nagoya-u.ac.jp}
\subjclass[2010]{11M06}
\keywords{Dirichlet $L$-function, derivative, zeros}
\thanks{The second named author was in part supported by
JSPS KAKENHI Grant Number 15J02325.}
\begin{abstract}
Y\i ld\i r\i m has classified zeros of the derivatives of Dirichlet
$L$-functions into trivial zeros, nontrivial zeros
and vagrant zeros.
In this paper we remove the possibility of vagrant zeros
for $L'(s,\chi)$ when the conductors are large to some extent.
Then we improve asymptotic formulas for the number of zeros of
$L'(s,\chi)$ in $\{s\in\C:\Rep(s)>0, |\Imp(s)|\leq T\}$.
We also establish analogues of Speiser's theorem,
which characterize the generalized Riemann hypothesis
for $L(s,\chi)$ in terms of zeros of $L'(s,\chi)$,
when the conductor is large.
\end{abstract}
\maketitle
\section{Introduction}
Let $\chi$ be a primitive Dirichlet character modulo $q>1$.
We denote the Dirichlet $L$-function attached to $\chi$
by $L(s,\chi)$.
In this paper we investigate zeros of $L'(s,\chi)$, the first
derivative of $L(s,\chi)$.
Previously, Y\i ld\i r\i m \cite{Yi} has investigated zeros of
the $k$-th derivative $L^{(k)}(s,\chi)$ of $L(s,\chi)$ for
any given $k\in\Z_{\geq 1}$.
He has shown that $L^{(k)}(s,\chi)$ does not vanish when $\Rep(s)$
is sufficiently large.
He has also obtained a zero-free region in a left
half-plane.
Strictly speaking, for any fixed $\varepsilon>0$
there exists $K>0$, which depends only on $k$ and $\varepsilon$,
such that $L^{(k)}(s,\chi)$ has no zeros
in $\{s=\sigma+it:|s|>q^K, \sigma<-\varepsilon, |t|>\varepsilon\}$.
Based on the above facts,
he classified zeros of $L^{(k)}(s,\chi)$ in the following way
(see \cite[\S 7]{Yi}):
\begin{itemize}
 \item trivial zeros, which are located on $\{\sigma+it:\sigma\leq -q^K,
       |t|\leq \varepsilon\}$,
\item vagrant zeros, which are located on $\{s=\sigma+it:|s|\leq q^K,
      \sigma\leq-\varepsilon\}$,
\item nontrivial zeros, which are located in $\{\sigma+it:\sigma>-\varepsilon\}$.
\end{itemize}

Loosely speaking, one of our main results is to remove the possibility of
vagrant zeros for $L'(s,\chi)$.
In order to state it precisely, we put
\[
 \Theta(\chi):=\sup\{\Rep(\rho):\rho\in\C, L(\rho,\chi)=0\}.
\]
It is easy to check that the following properties hold:
\begin{itemize}
 \item $1/2\leq\Theta(\chi)\leq 1$.
 \item $\Theta(\overline{\chi})=\Theta(\chi)$.
 \item For each primitive Dirichlet character $\chi$,
the generalized Riemann hypothesis (GRH, in short) for $L(s,\chi)$
is equivalent to $\Theta(\chi)=1/2$.
\end{itemize}
The following theorem says that
$L'(s,\chi)$ does not vanish apart from $\Imp(s)=0$ on a left half-plane:
\begin{Theorem}\label{Thm1}
Let $\chi$ be a primitive Dirichlet character modulo $q>1$.
Then $L'(s,\chi)$ has no zeros
on $s\in \mathcal{D}_1(\chi)\cup \mathcal{D}_2(\chi)$,
where
\begin{align*}
\mathcal{D}_1(\chi)
&=\left\{\sigma+it:\sigma\leq
1-\Theta(\chi),~|t|\geq\frac{6}{\log q}\right\}
\setminus\{\rho\in\C:L(\rho,\chi)=0\},\\
\mathcal{D}_2(\chi)
&=\left\{\sigma+it:\sigma\leq -q^2,~|t|\geq\frac{12}{\log|\sigma|}\right\}.
\end{align*}
\end{Theorem}
\begin{Remark}
The constants $6$ and $12$ in $\mathcal{D}_1(\chi)$ and
$\mathcal{D}_2(\chi)$ can be replaced by smaller constants.
\end{Remark}
Let $\chi$ be a primitive Dirichlet character modulo $q>1$.
We determine $\kappa\in\{0,1\}$ such that $\chi(-1)=(-1)^{\kappa}$.
We recall that zeros of $L(s,\chi)$ on $\Rep(s)\leq 0$
are located at $s=-2j-\kappa$ for any $j\in\Z_{\geq 0}$.
These are called trivial zeros of $L(s,\chi)$.
The following theorem says that there is a unique zero
of $L'(s,\chi)$ corresponding to the trivial zero $s=-2j-\kappa$
of $L(s,\chi)$ for each $j\in\Z_{\geq 1}$.
\begin{Theorem}\label{Thm2}
Retain the above setting.
Then for each $j\in\Z_{\geq 1}$ the following assertions hold:
\begin{enumerate}
 \item[(1)] There is a unique zero of $L'(s,\chi)$ in the strip
 $s\in\{\sigma+it:-2j-\kappa-1<\sigma<-2j-\kappa+1,~t\in\R\}$.
\item[(2)] $L'(s,\chi)\neq 0$ on $\Rep(s)=-2j-\kappa+1$.
\end{enumerate}
\end{Theorem}
\begin{Remark}
In \cite[\S 5]{Yi}, Y\i ld\i r\i m has shown that
there exists $J=J_\chi>0$ such that the above assertions hold for any $j\geq J$.

 In the case of the Riemann zeta-function $\zeta(s)$, similar results
can be found in \cite[Theorem 9]{LM} and \cite{Spi}.
\end{Remark}
Retain the setting as in Theorem \ref{Thm2}.
For $j\in \Z_{\geq 1}$ we denote the zero of $L'(s,\chi)$
in $\{\sigma+it:-2j-\kappa-1<\sigma<-2j-\kappa+1\}$
by $\alpha_j(\chi)$.
According to Theorem \ref{Thm3} below, the zero $\alpha_j(\chi)$
is close to the trivial zero $s=-2j-\kappa$
of $L(s,\chi)$ when $jq$ is large.
\begin{Theorem}\label{Thm3}
 Retain the above setting.
Then we have
\[
 \alpha_j(\chi)=-2j-\kappa+O\left(\frac{1}{\log(jq)}\right),
\]
where the implied constant is absolute.
\end{Theorem}
Concerning zeros of $L'(s,\chi)$ near the trivial zero $s=-\kappa$
of $L(s,\chi)$, we see the following:
\begin{Theorem}\label{Thm4}
Let $\chi$ be a primitive Dirichlet character modulo $q$.
Then,
\begin{enumerate}
 \item If $\chi(-1)=1$ and $q\geq 7$, then $L'(s,\chi)$ has no zeros
on $\{\sigma+it:-1\leq\sigma\leq 0,~t\in\R\}$.
\item If $\chi(-1)=-1$ and $q\geq 23$, then $L'(s,\chi)$ has a unique zero
on $\{\sigma+it:-2\leq\sigma\leq 0,~t\in\R\}$.
\end{enumerate}
\end{Theorem}
\begin{Remark}
When $\chi(-1)=1$ and $q\geq 216$, a zero of $L'(s,\chi)$ corresponding to $s=0$
is expected to appear in $\{\sigma+it:0<\sigma<1/2,~t\in\R\}$.
In fact, Y\i ld\i r\i m \cite[Theorem 1]{Yi} has shown that
there is a unique zero of $L'(s,\chi)$ in $0\leq\Rep(s)<1/2$
and this zero is located near $s=0$,
assuming GRH for $L(s,\chi)$.
\end{Remark}
Based on Theorems \ref{Thm1}--\ref{Thm4}, we reconsider
Y\i ld\i r\i m's classification on zeros of $L'(s,\chi)$.
Except for the finite number of primitive Dirichlet characters $\chi$,
there is a one-to-one correspondence between zeros of $L'(s,\chi)$ in $\Rep(s)<0$
and trivial zeros of $L(s,\chi)$ in $\Rep(s)<0$,
thanks to Theorems \ref{Thm2} and \ref{Thm4}.
For the excluded characters $\chi$,
there are at most a finite number of zeros of $L'(s,\chi)$
on $-1\leq\Rep(s)\leq 0$ when $\chi$ is even and
on $-2\leq\Rep(s)\leq 0$ when $\chi$ is odd, thanks to Theorem \ref{Thm1}.
Keeping these in mind, we propose to modify Y\i ld\i r\i m's
classification on zeros of $L'(s,\chi)$
in the following way:
\begin{itemize}
 \item trivial zeros, which are located on $\Rep(s)\leq 0$,
 \item nontrivial zeros, which are located in $\Rep(s)> 0$.
\end{itemize}
\begin{Remark}
Let $\chi(-1)=1$ and $q\geq 216$.
 As was mentioned in Remark of Theorem \ref{Thm4},
we expect that a zero of $L'(s,\chi)$, which corresponds to the trivial zero $s=0$
of $L(s,\chi)$, appears in $0<\Rep(s)<1/2$.
Though this zero is nontrivial according to
the above classification,
we may say that this zero is trivial.
However, since this does not affect the present paper,
we do not discuss it.
\end{Remark}

We turn to the distribution of nontrivial zeros for $L'(s,\chi)$.
Let $\chi$ be a primitive Dirichlet character modulo $q>1$.
We put
\[
 m:=\min\{n\in\Z_{\geq 2}:\chi(n)\neq 0\}.
\]
We easily see $m=\min\{n\in\Z_{\geq 2}:\gcd(n,q)=1\}$.
This together with the prime number theorem yields $m\ll \log q$.
For $T>0$ we denote by $N_1(T,\chi)$  the number of zeros
of $L'(s,\chi)$ on $\{\sigma+it:\sigma>0, -T\leq t\leq T\}$,
counted with multiplicity.
We note that our $N_1(T,\chi)$ differs from Y\i ld\i r\i m's in
\cite[\S 6]{Yi} slightly,
because he counted vagrant zeros of $L'(s,\chi)$ as well as nontrivial
zeros of $L'(s,\chi)$.
With this notation we have
\begin{Theorem}\label{Thm5}
Retain the above notation. Then for $T\geq 2$ we have
\[
 N_1(T,\chi)=\frac{T}{\pi}\log\frac{qT}{2\pi m}
-\frac{T}{\pi}
+O(m^{1/2}\log(qT)),
\]
where the implied constant is absolute.
\end{Theorem}
\begin{Remark}
 It may be interesting to compare our results with \cite[Theorem 4]{Yi}.
Our error term estimate is considerably smaller than
Y\i ld\i r\i m's with respect to the conductor $q$.

See \cite{Be} for a similar formula for nontrivial zeros of $\zeta'(s)$.
\end{Remark}

We also give an asymptotic formula on the horizontal distribution of
nontrivial zeros of $L'(s,\chi)$,
which is an analogue of \cite[Theorem 10]{LM}.
\begin{Theorem}\label{Thm6}
Retain the notation. Then for $T\geq 2$ it holds that
\begin{align*}
 \sum_{\begin{subarray}{c}
        \rho'=\beta'+i\gamma',\\
        \beta'>0, -T\leq\gamma'\leq T
       \end{subarray}}
\left(\beta'-\frac{1}{2}\right)
&=\frac{T}{\pi}\log\log\frac{qT}{2\pi}
+\frac{T}{\pi}\left(\frac{1}{2}\log m-\log\log m\right)\\
&-\frac{2}{q}\li\left(\frac{qT}{2\pi}\right)
+O(m^{1/2}\log(qT)),
\end{align*}
where $\rho'=\beta'+i\gamma'$ runs over all zeros of $L'(s,\chi)$
satisfying $\beta'>0$ and $-T\leq\gamma'\leq T$,
counted with multiplicity and
\[
 \li(x):=\int_2^x\frac{du}{\log u}.
\]
Here the implied constant is absolute.
\end{Theorem}
Finally we consider zeros of $L'(s,\chi)$ in $0<\Rep(s)<1/2$.
The old celebrated theorem of Speiser \cite{Sp} says that
the Riemann hypothesis is equivalent to having no non-real zeros
of $\zeta'(s)$ in $\Rep(s)<1/2$.
In this paper we give an analogue of Speiser's theorem for each
$L(s,\chi)$ with large $q$,
which characterize GRH in terms of
nontrivial zeros of $L'(s,\chi)$.
\begin{Theorem}\label{Thm8}
 Let $\chi$ be an even primitive Dirichlet character with
conductor $q\geq 216$.
Then the following (i) and (ii)
are equivalent:
\begin{enumerate}
\item[(i)] $L(s,\chi)\neq 0$ in $0<\Rep(s)<1/2$.
\item[(ii)] There is a unique zero of $L'(s,\chi)$ in $0<\Rep(s)<1/2$.
\end{enumerate}
\end{Theorem}
\begin{Theorem}\label{Thm9}
Let $\chi$ be an odd primitive Dirichlet character
with conductor $q\geq 23$.
Then the following conditions (i) and (ii) are equivalent:
\begin{enumerate}
\item[(i)] $L(s,\chi)\neq 0$ in $0<\Rep(s)<1/2$.
\item[(ii)] There are no zeros of $L'(s,\chi)$ in $0<\Rep(s)<1/2$.
\end{enumerate}
\end{Theorem}
\begin{Remark}
 The implications (i)$\Longrightarrow$(ii) in Theorems \ref{Thm8}
and \ref{Thm9} have been obtained by  Y\i ld\i r\i m
\cite[Theorem 1]{Yi}.
Our contribution in this paper is to establish the implications
(ii)$\Longrightarrow$(i).
\end{Remark}
Theorems \ref{Thm8} and \ref{Thm9} give us analogues of
Speiser's theorem only when $q$ are large to some extent.
However, it is highly possible to formulate similar assertions
when $q$ are small,
by investigating the variation of $\arg{(L'/L)}(s,\chi)$ on $\Rep(s)=0$
and $\Rep(s)=1/2$ through numerical calculations.

In this paper we also show the quantitative version of
Speiser's theorem for $L'(s,\chi)$, which is stated as follows:
\begin{Theorem}\label{Thm7}
 For $T\geq 2$ we have
\begin{equation}\label{eq:SC1}
 N^{-}(T,\chi)=N^{-}_1(T,\chi)+O(m^{1/2}\log(qT)),
\end{equation}
where $N^{-}(T,\chi)$ and $N^{-}_1(T,\chi)$ are the number of zeros
of $L(s,\chi)$ and $L'(s,\chi)$ on $\{\sigma+it:0<\sigma<1/2, -T\leq
 t\leq T\}$ respectively,
where zeros are counted with multiplicity.
Here the implied constant is absolute.
\end{Theorem}
\begin{Remark}
In \cite[Theorem 1.2]{GS}
Garunk\v{s}tis and \v{S}im\.{e}nas
have obtained (\ref{eq:SC1}) for fixed $\chi$.
The new element of this paper is to give uniform estimates
with respect to $\chi$.
\end{Remark}

This paper is organized as follows.
In \S 2 we treat zeros of $L'(s,\chi)$ on a left half-plane,
which includes the proof of Theorems \ref{Thm1}--\ref{Thm4}.
In \S 3 we deal with the distribution of zeros of $L'(s,\chi)$ in
$\Rep(s)>0$ and show Theorems \ref{Thm5} and \ref{Thm6}.
In \S 4 we show Theorems \ref{Thm8}--\ref{Thm7}.

In this paper we need some numerical computations.
We use PARI/GP 2.7.3.
\subsection*{Notation}
Throughout this paper let $\chi$ be primitive Dirichlet
characters modulo $q>1$.
For $\chi$ we put $m:=\min\{n\in\Z_{\geq 2}:\chi(n)\neq 0\}$.
The number $\kappa\in\{0,1\}$ is determined by $\chi(-1)=(-1)^{\kappa}$.
We denote nontrivial zeros of $L(s,\chi)$
(i.e., zeros in $\{\sigma+it:0<\sigma<1\}$) by
$\rho=\beta+i\gamma$
and zeros of $L'(s,\chi)$ by $\rho'=\beta'+i\gamma'$.
We put $\Theta(\chi):=\sup_{\rho}\Rep(\rho)$.
$c_E$ is the Euler--Mascheroni constant.
\section{Zeros of $L'(s,\chi)$ on a left half-plane}

In this section we prove Theorems \ref{Thm1}--\ref{Thm4}.
Our strategy for showing these is to seek paths and regions
on which $\Rep(L'/L)(s,\chi)$ is negative.

First of all we show Theorem \ref{Thm1}.
It suffices to show the following:
 \begin{Proposition}\label{Prop:negative1}
Keep the notation in Theorem \ref{Thm1}.
Then $\Rep(L'/L)(s,\chi)<0$ holds on $s\in\mathcal{D}_1(\chi)\cup\mathcal{D}_2(\chi)$.
\end{Proposition}
In order to show Proposition \ref{Prop:negative1},
we start with the logarithmic derivative of
the Hadamard product expression for $L(s,\chi)$.
It is given by (see \cite[Corollary 10.18]{MV}):
\begin{equation}\label{eq:Hadamard}
 \frac{L'}{L}(s,\chi)=
B(\chi)-\frac{1}{2}\frac{\Gamma'}{\Gamma}\left(\frac{s+\kappa}{2}\right)
-\frac{1}{2}\log\frac{q}{\pi}+\sum_{\rho}\left(\frac{1}{s-\rho}+\frac{1}{\rho}\right),
\end{equation}
where $B(\chi)$ is a constant depending only on $\chi$, which satisfies
\[
 \Rep(B(\chi))=-\sum_{\rho}\Rep\frac{1}{\rho}.
\]
 Taking the real part on (\ref{eq:Hadamard}),
\begin{equation}\label{eq:Hadamard3}
 \Rep\frac{L'}{L}(s,\chi)
=-\frac{1}{2}\log\frac{q}{\pi}
-\frac{1}{2}\Rep\frac{\Gamma'}{\Gamma}\left(\frac{s+\kappa}{2}\right)
+\sum_{\rho=\beta+i\gamma}
\frac{\sigma-\beta}{|s-\rho|^2}.
\end{equation}
We see from  the definition of $\Theta(\chi)$ and the functional equation
that $\beta\geq 1-\Theta(\chi)$ holds for any $\rho$.
Thus we find $\sigma-\beta\leq 0$ if $\sigma\leq 1-\Theta(\chi)$.
This says that the sum over nontrivial zeros is nonpositive
on $\{s=\sigma+it:\sigma\leq 1-\Theta(\chi), L(s,\chi)\neq 0\}$.
Thus, on this region the following inequality holds:
\begin{equation}\label{eq:Hadamard2}
 \Rep\frac{L'}{L}(s,\chi)
\leq
-\frac{1}{2}\log\frac{q}{\pi}
-\frac{1}{2}\Rep\frac{\Gamma'}{\Gamma}\left(\frac{s+\kappa}{2}\right).
\end{equation}

To estimate the last term on (\ref{eq:Hadamard2}),
we show the following inequality:
\begin{Lemma}
 For $z=x+iy$ with $x\in\R$ and $y\in\R\setminus\{0\}$ we have
\begin{equation}\label{eq:Gammaineq0}
 \Rep\frac{\Gamma'}{\Gamma}(z)\geq\log|z|-\frac{\pi}{2|y|}.
\end{equation}
\end{Lemma}
\begin{proof}
 We start with the following expression
for $(\Gamma'/\Gamma)(z)$ (see \cite[(C.10) in p.522]{MV}):
\begin{equation}\label{eq:HadamardG}
 \frac{\Gamma'}{\Gamma}(z)
=-c_E-\sum_{n=0}^{\infty}
\left(\frac{1}{n+z}-\frac{1}{n+1}\right).
\end{equation}
Suppose $z=x+iy\in\C\setminus(-\infty,0]$.
Applying the Euler--Maclaurin summation formula
(that is, \cite[Theorem B.5 when $K=1$]{MV}),
we have
\begin{equation}\label{eq:EM}
\lim_{N\to\infty}\left(\sum_{n=0}^N\frac{1}{n+z}-\Log(N+z)\right)
=-\Log z+\frac{1}{2z}-\int_0^{\infty}\frac{u-[u]-\frac{1}{2}}{(u+z)^2}du,
\end{equation}
where $\Log z$ is the principal logarithmic branch of $z$.
By the definition of $c_E$ we have
\begin{equation}\label{eq:Euler}
 \lim_{N\to\infty}
\left(
\sum_{n=0}^N\frac{1}{n+1}-\log N
\right)=c_E.
\end{equation}
We subtract (\ref{eq:Euler}) from (\ref{eq:EM}) and apply it
to (\ref{eq:HadamardG}).
In consequence we have
\begin{equation}\label{eq:diGamma}
 \frac{\Gamma'}{\Gamma}(z)=\Log z-\frac{1}{2z}
+\int_0^{\infty}\frac{u-[u]-\frac{1}{2}}{(u+z)^2}du.
\end{equation}
Taking the real part, we obtain
\begin{equation}\label{eq:ReGamma}
 \Rep\frac{\Gamma'}{\Gamma}(z)
=\log|z|-\frac{1}{2}\Rep\frac{1}{z}
+\Rep\int_0^{\infty}\frac{u-[u]-\frac{1}{2}}{(u+z)^2}du.
\end{equation}

We consider the case $x\geq 0$ and $y\neq 0$.
Then we have
\begin{gather*}
 \left|\Rep\frac{1}{z}\right|\leq\frac{1}{|z|}\leq\frac{1}{|y|},\\
\left|\Rep\int_0^{\infty}\frac{u-[u]-\frac{1}{2}}{(u+z)^2}du\right|
\leq\frac{1}{2}\int_0^{\infty}\frac{du}{|u+z|^2}
\leq\frac{1}{2}\int_0^{\infty}\frac{du}{u^2+y^2}
=\frac{\pi}{4|y|}.
\end{gather*}
Inserting these into (\ref{eq:ReGamma}), we obtain
\[
 \Rep\frac{\Gamma'}{\Gamma}(z)\geq\log|z|-
\left(\frac{1}{2}+\frac{\pi}{4}\right)\frac{1}{|y|}.
\]
This yields the result when $x\geq 0$.

We consider the case $x<0$ and $y\neq 0$.
In this case $\Rep(1/z)$ is negative.
In addition, a standard estimate gives
\[
 \left|\Rep\int_0^{\infty}\frac{u-[u]-\frac{1}{2}}{(u+z)^2}du\right|
\leq\frac{1}{2}\int_{-\infty}^{\infty}\frac{du}{u^2+y^2}
=\frac{\pi}{2|y|}.
\]
Applying these to (\ref{eq:ReGamma}), we obtain the result
when $x<0$.
\end{proof}
\begin{proof}[Proof of Proposition \ref{Prop:negative1}]
Suppose that $s=\sigma+it$ satisfies
$\sigma\leq 1-\Theta(\chi)$, $t\neq 0$ and $L(s,\chi)\neq 0$.
Inserting (\ref{eq:Gammaineq0}) into (\ref{eq:Hadamard2}),
\[
 \Rep\frac{L'}{L}(s,\chi)
\leq-\frac{1}{2}\log\left|\frac{q(s+\kappa)}{2\pi}\right|
+\frac{\pi}{2|t|}.
\]
Thus we obtain the following inequalities:
\begin{align}
&\Rep\frac{L'}{L}(s,\chi)
\leq-\frac{1}{2}\log\frac{q|t|}{2\pi}
+\frac{\pi}{2|t|},\label{eq:Lineq1}\\
&\Rep\frac{L'}{L}(s,\chi)
\leq-\frac{1}{2}\log\frac{q|\sigma+\kappa|}{2\pi}
+\frac{\pi}{2|t|}.\label{eq:Lineq2}
\end{align}
If $|t|\geq 6/\log q$, then (\ref{eq:Lineq1}) is bounded above by
\begin{equation}\label{eq:Lineq3}
 \leq -\frac{1}{2}\left(1-\frac{\pi}{6}\right)\log q
+\frac{1}{2}\log\log q+\frac{1}{2}\log\frac{\pi}{3}.
\end{equation}
 We can easily check that $x^{\alpha}/\log x\geq \alpha e$
 holds for $x>1$ and $\alpha>0$.
 Taking the logarithm, we have $\alpha\log x-\log\log x\geq 1+\log\alpha$.
 Applying this with $\alpha=1-\frac{\pi}{6}$ and $x=q$,
 we see that (\ref{eq:Lineq3}) is
\[
 \leq -\frac{1}{2}\log\left(1-\frac{\pi}{6}\right)-\frac{1}{2}
+\frac{1}{2}\log\frac{\pi}{3}<-0.106.
\]
This confirms that $\Rep(L'/L)(s,\chi)$ is negative
on $s\in\mathcal{D}_1(\chi)$.
It is easy to check from (\ref{eq:Lineq2}) that
$\Rep(L'/L)(s,\chi)$ is negative on $s\in\mathcal{D}_2(\chi)$,
whose details are omitted.
\end{proof}
\begin{proof}[Proof of Theorem \ref{Thm1}]
 Theorem \ref{Thm1} is an immediate consequence of Proposition
\ref{Prop:negative1}.
\end{proof}
Next we prove Theorem \ref{Thm2}.
For this purpose we show
\begin{Proposition}\label{Prop:negative2}
Keep the notation in Theorem \ref{Thm2}.
 Then for each $j\in\Z_{\geq 1}$, the inequality
$\Rep(L'/L)(s,\chi)\leq -10^{-4}(<0)$ holds
on $\Rep(s)=-2j-\kappa+1$.

\end{Proposition}
\begin{proof}
We start with the logarithmic derivative of the functional
equation for $L(s,\chi)$, which is written as
(see \cite[p.352]{MV})
\begin{equation}\label{eq:fteq}
 \frac{L'}{L}(s,\chi)
=-\frac{L'}{L}(1-s,\overline{\chi})-\log\frac{q}{2\pi}
-\frac{\Gamma'}{\Gamma}(1-s)+
\frac{\pi}{2}\cot\left(\frac{\pi(s+\kappa)}{2}\right).
\end{equation}
We take $j\in\Z_{\geq 1}$ and $t\in\R$ and put $s=-2j-\kappa+1+it$
on (\ref{eq:fteq}).
Then we take the real part.
Since the last term on (\ref{eq:fteq}) is purely imaginary
on $\Rep(s)=-2j-\kappa+1$,
we have
\begin{equation}\label{eq:fteq2}
\begin{aligned}
&\Rep\frac{L'}{L}(-2j-\kappa+1+it,\chi)\\
&=-\Rep\frac{L'}{L}(2j+\kappa-it,\overline{\chi})-\log\frac{q}{2\pi}
-\Rep\frac{\Gamma'}{\Gamma}(2j+\kappa-it).
\end{aligned}
\end{equation}
Firstly we treat the first term on the right.
Taking the logarithmic derivative of the Euler product
for $L(s,\overline{\chi})$,
in $\Rep(s)>1$ we have
\begin{equation}\label{eq:LDEP}
 \frac{L'}{L}(s,\overline{\chi})
=-\sum_{p:\text{primes}}
\frac{\overline{\chi}(p)\log p}{p^s-\overline{\chi}(p)}.
\end{equation}
Thus we put $s=2j+\kappa-it$ and estimate it trivially,
so that
\begin{equation}\label{eq:Dineq}
 \left|\Rep\frac{L'}{L}(2j+\kappa-it,\overline{\chi})\right|
\leq
\sum_{p\nmid q}
\frac{\log p}{p^{2j+\kappa}-1}.
\end{equation}
Next we deal with the last term on (\ref{eq:fteq2}).
It follows from (\ref{eq:HadamardG}) that
for $z=x+iy$ with $x>0$, $y\in\R$
\begin{equation}\label{eq:Gammaineq}
 \Rep\frac{\Gamma'}{\Gamma}(z)-\frac{\Gamma'}{\Gamma}(x)
=y^2\sum_{n=0}^{\infty}\frac{1}{(n+x)\{(n+x)^2+y^2\}}\geq 0.
\end{equation}
Thus, putting $x=2j+\kappa$ and $y=-t$,
we see that
\begin{equation}\label{eq:Gineq}
 \Rep\frac{\Gamma'}{\Gamma}(2j+\kappa-it)\geq\frac{\Gamma'}{\Gamma}(2j+\kappa)
=-c_E+\sum_{a=1}^{2j+\kappa-1}\frac{1}{a}.
\end{equation}
Applying (\ref{eq:Dineq}) and (\ref{eq:Gineq}) to (\ref{eq:fteq2}),
we obtain
\begin{equation}\label{eq:LDineq}
\Rep\frac{L'}{L}(-2j-\kappa+1+it)\leq A(q,\kappa;j)+B(q,\kappa;j),
\end{equation}
where
\begin{align*}
&A(q,\kappa;j):=c_E-\sum_{a=1}^{2j+\kappa-1}\frac{1}{a}-\log\frac{q}{2\pi},\\
&B(q,\kappa;j):=\sum_{p\nmid q}\frac{\log p}{p^{2j+\kappa}-1}.
\end{align*}
 To compute $B(q,\kappa;j)$ numerically, we note that
\begin{equation}\label{eq:Error1}
 \sum_{p>N}\frac{\log p}{p^{\sigma}-1}
  \leq\frac{N}{N^{\sigma}-1}
  \left(\frac{\log N}{\sigma-1}+\frac{1}{(\sigma-1)^2}\right)
\end{equation}
holds for $\sigma>1$ and $N\in\Z_{\geq 3}$.
In fact, the left-hand side of (\ref{eq:Error1}) is
\[
 \leq\sum_{p>N}\frac{\log p}{p^{\sigma}-(p/N)^{\sigma}}
 =\frac{N^{\sigma}}{N^{\sigma}-1}\sum_{p>N}\frac{\log p}{p^{\sigma}}.
\]
 Plainly this is estimated as
\[
 \leq\frac{N^{\sigma}}{N^{\sigma}-1}
 \sum_{n=N+1}^{\infty}\frac{\log n}{n^{\sigma}}
 \leq\frac{N^{\sigma}}{N^{\sigma}-1}
 \int_N^{\infty}\frac{\log u}{u^{\sigma}}du.
\]
Calculating the integral, we obtain (\ref{eq:Error1}).

We go back to (\ref{eq:LDineq}).
Firstly we consider the case $\kappa=1$.
Since $q\geq 3$ and $j\geq 1$, we have
\begin{align*}
&A(q,1;j)\leq c_E-\frac{3}{2}-\log\frac{3}{2\pi}<-0.183,\\
&B(q,1;j)\leq \sum_{p}\frac{\log p}{p^3-1}<0.174.
\end{align*}
Here we used (\ref{eq:Error1}) with $\sigma=3$ and $N=10$.
This implies the desired result when $\kappa=1$.

We treat the case $\kappa=0$.
We note that $\kappa=0$ implies $q\geq 5$
and that there are no primitive Dirichlet characters modulo $6$.
When $q\geq 8$ and $j\geq 1$, we see from (\ref{eq:Error1})
with $\sigma=2$ and $N=100$ that
\begin{align*}
&A(q,0;j)\leq c_E-1-\log\frac{8}{2\pi}<-0.66,\\
&B(q,0;j)\leq\sum_{p}\frac{\log p}{p^2-1}<0.62.
\end{align*}
When $q=7$ and $j\geq 1$, by (\ref{eq:Error1}) with $\sigma=2$
and $N=10^5$
we have
\begin{align*}
&A(7,0;j)\leq c_E-1-\log\frac{7}{2\pi}<-0.53,\\
&B(7,0;j)\leq\sum_{p\neq 7}\frac{\log p}{p^2-1}<0.5296.
\end{align*}
Thus we obtain the desired result when $q\geq 7$ and $\kappa=0$.

It remains to show the assertion in the case $q=5$ and $\kappa=0$.
Then $\chi$ is determined uniquely and given in terms of the Kronecker
symbol by $\chi(n)=\chi_5(n):=\left(\frac{5}{n}\right)$.
For $j\geq 2$ we see from (\ref{eq:Error1}) with
$\sigma=4$ and $N=10$ that for $j\geq 2$
\begin{align*}
&A(5,0;j)\leq c_E-\frac{11}{6}-\log\frac{5}{2\pi}<-1,\\
&B(5,0;j)\leq \sum_{p\neq 5}\frac{\log p}{p^4-1}<0.07.
\end{align*}
This implies the desired result in the case $j\geq 2$.
We consider the case $j=1$.
Since $\chi_5$ is real, $\Rep(L'/L)(-1+it,\chi_5)=\Rep(L'/L)(-1-it,\chi_5)$ holds
for $t\in\R$.
Thus it suffices to show that $\Rep(L'/L)(-1+it,\chi_5)$
is negative for $t\geq 0$.
For this purpose we use (\ref{eq:fteq2}) with $\chi=\chi_5$ and $j=1$:
\begin{equation}\label{eq:fteq3}
 \Rep\frac{L'}{L}(-1+it,\chi_5)=-\Rep\frac{L'}{L}(2-it,\chi_5)
-\log\frac{5}{2\pi}-\Rep\frac{\Gamma'}{\Gamma}(2-it).
\end{equation}

First of all we consider the case $t\geq 3/2$.
By the same manner as (\ref{eq:Dineq}) we have
\begin{equation}\label{eq:Dineq2}
 \left|\Rep\frac{L'}{L}(2-it,\chi_5)\right|\leq
\sum_{p\neq 5}\frac{\log p}{p^2-1}<0.51.
\end{equation}
Here we used (\ref{eq:Error1}) with $\sigma=2$ and $N=1000$.
It is easy to see that
the right-hand side of (\ref{eq:Gammaineq}) is monotonically decreasing
on $y\leq -3/2$,
so that
\begin{equation}\label{eq:digammaineq}
 \Rep\frac{\Gamma'}{\Gamma}(2-it)
\geq 1-c_E+
\frac{9}{4}\sum_{n=0}^{\infty}\frac{1}{(n+2)\{(n+2)^2+9/4\}}
>0.75
\end{equation}
holds on $t\geq 3/2$.
Here in the first inequality we used $(\Gamma'/\Gamma)(2)=1-c_E$
and in the last inequality we ignored the sum over $n>100$. 
Inserting (\ref{eq:Dineq2}), (\ref{eq:digammaineq})
and $-\log(5/2\pi)<0.23$ into (\ref{eq:fteq3}),
we see that $\Rep(L'/L)(-1+it,\chi_5)<-0.01$ for $t\geq 3/2$.

Finally we consider the case $0\leq t<3/2$.
We deal with the first term on the right-hand side of (\ref{eq:fteq3}).
Using (\ref{eq:LDEP}) and (\ref{eq:Error1}) with $N=1000$,
we compute $\Rep(L'/L)(s,\chi)$ numerically at some points
on $\Rep(s)=2$ as follows:
\begin{equation}\label{eq:numerical}
\begin{gathered}
\frac{L'}{L}(2,\chi_5)>0.27,\phantom{MM}
\Rep\frac{L'}{L}\left(2-\frac{i}{2},\chi_5\right)>0.24,\\
\Rep\frac{L'}{L}(2-i,\chi_5)>0.16,
\phantom{MM}
\Rep\frac{L'}{L}\left(2-\frac{5}{4}i,\chi_5\right)>0.11,\\
\Rep\frac{L'}{L}\left(2-\frac{11}{8}i,\chi_5\right)>0.08,
\phantom{MM}
\Rep\frac{L'}{L}\left(2-\frac{3}{2}i,\chi_5\right)>0.06.
\end{gathered}
\end{equation}
We note that for $t\in\R$ and $t_0\in\{0,1/2,1,5/4,11/8,3/2\}$
\begin{equation}\label{eq:Lintegral}
 \Rep\frac{L'}{L}(2-it,\chi_5)=\Rep\frac{L'}{L}(2-it_0,\chi_5)
+\Imp\int_{t_0}^{t}
\left(\frac{L'}{L}\right)'(2-iv,\chi_5)dv.
\end{equation}
Numerical computation gives that for $v\in\R$
\begin{equation}\label{eq:numerical2}
 \left|\left(\frac{L'}{L}\right)'(2-iv,\chi_5)\right|
\leq\sum_{p\neq 5}\frac{p^{-2}(\log p)^2}{(1-p^{-2})^2}
<0.79.
\end{equation}
Here in the last inequality we computed the sum numerically
up to $10^4$ and we used
\[
 \sum_{p>10^4}\frac{p^{-2}(\log p)^2}{(1-p^{-2})^2}
\leq\frac{1}{(1-10^{-8})^2}\int_{10^4}^{\infty}\frac{(\log u)^2}{u^2}du
<0.011.
\]
We see from (\ref{eq:numerical}), (\ref{eq:Lintegral}) and (\ref{eq:numerical2})
that $\Rep(L'/L)(2-it,\chi_5)>0$ for $0\leq t<3/2$.
This together with (\ref{eq:Gammaineq}) and (\ref{eq:fteq3}) yields
\[
 \Rep\frac{L'}{L}(-1+it,\chi_5)<-\log\frac{5}{2\pi}-1+c_E<-0.1
\]
for $0\leq t<3/2$.
This completes the proof.
\end{proof}
\begin{proof}[Proof of Theorem \ref{Thm2}]
Let $j\in\Z_{\geq 1}$.
Proposition \ref{Prop:negative2} implies that $L'(s,\chi)$
does not vanish on $\Rep(s)=-2j-\kappa-1$.
We shall show that $L'(s,\chi)$ has a unique zero in the strip
$-2j-\kappa-1<\Rep(s)<-2j-\kappa+1$.
We take the path determined by the rectangle
with vertices at $-2j-\kappa\pm 1\pm 1000i$.
By Propositions \ref{Prop:negative1} and \ref{Prop:negative2}
we find that $\Rep(L'/L)(s,\chi)$ is negative on the path.
Thus the argument principle gives that the number of zeros
of $L'(s,\chi)$ inside the path equals that of $L(s,\chi)$.
Since $L(s,\chi)$ has a unique zero $s=-2j-\kappa$
inside the path,
there is a unique zero of $L'(s,\chi)$ inside the path.
Combining this with Theorem \ref{Thm1}, we reach the first claim
of Theorem \ref{Thm2}.
\end{proof}
Next we prove Theorem \ref{Thm3}.
\begin{proof}[Proof of Theorem \ref{Thm3}]
We take $\varepsilon\in(0,1/2)$.
Let $\mathcal{C}=\mathcal{C}_{j,\varepsilon}$
be the path determined by the circle
of the center $-2j-\kappa$ and the radius $\varepsilon$.
Then it is easy to see from (\ref{eq:fteq}) and Stirling's formula
that
\[
\Rep\frac{L'}{L}(s,\chi)=-\log(jq)+\Rep\frac{1}{\eta}+O(1)
\]
holds on $s=-2j-\kappa+\eta\in\mathcal{C}$,
where the implied constant is absolute.
Suppose that $jq$ is sufficiently large and
we choose $\varepsilon=2/\log(jq)$.
Then we find that $\Rep(L'/L)(s,\chi)$ is negative on $s\in\mathcal{C}$.
Thus the argument principle says that there is a unique zero
of $L'(s,\chi)$ inside $\mathcal{C}$ thanks to the trivial
zero $s=-2j-\kappa$ of $L(s,\chi)$.
Since the zero of $L'(s,\chi)$ inside $\mathcal{C}$ coincides
with $\alpha_j(\chi)$,
we obtain $|\alpha_j(\chi)+2j+\kappa|<2/\log(jq)$.
This completes the proof.
\end{proof}
Next we show Theorem \ref{Thm4}.
\begin{proof}[Proof of Theorem \ref{Thm4}]
We consider the case $\chi(-1)=1$.
We will check that the right-hand side of (\ref{eq:Hadamard2})
is negative on $s\in\R\setminus\{0\}$ and near $s=0$
with $\Rep(s)\leq 0$.
Let $t\in\R\setminus\{0\}$.
Then we have
\[
 \Rep\frac{\Gamma'}{\Gamma}\left(\frac{it}{2}\right)
=\Rep\frac{\Gamma'}{\Gamma}\left(1+\frac{it}{2}\right)
\geq\frac{\Gamma'}{\Gamma}(1)=-c_E.
\]
Applying this to (\ref{eq:Hadamard2}), we obtain
\[
 \Rep\frac{L'}{L}(it,\chi)\leq-\frac{1}{2}\log\frac{q}{\pi}+\frac{1}{2}c_E.
\]
Next suppose $0<|s|\leq 1/1000$ and $\Rep(s)\leq 0$.
Then
\[
 \frac{\Gamma'}{\Gamma}\left(\frac{s}{2}\right)
=\frac{\Gamma'}{\Gamma}\left(1+\frac{s}{2}\right)-\frac{2}{s}
=-\frac{2}{s}-c_E+O(|s|).
\]
This together with $\Rep(1/s)\leq 0$ gives
\begin{equation}\label{eq:logderineq}
 \Rep\frac{L'}{L}(s,\chi)\leq -\frac{1}{2}\log\frac{q}{\pi}+\frac{1}{2}c_E
+O(|s|)
\end{equation}
for $s\in\C$ satisfying $0<|s|\leq 1/1000$ and $\Rep(s)\leq 0$.
Here we note that $-\frac{1}{2}\log(q/\pi)+c_E/2$ is negative
if $q>\pi e^{c_E}=5.59\ldots$.
Suppose $q\geq 7$.
Then there exists $\delta_0>0$ such that the right-hand side
of (\ref{eq:logderineq}) is negative on $\{s\in\C:0<|s|\leq\delta_0,~
 \Rep(s)\leq 0\}$.
We take any $\delta\in(0,\delta_0]$.
We take the contour $\mathcal{C}$ determined by the rectangle
with vertices at $-1\pm 1000i$, $\pm 1000i$ with a small
left-semicircular indentation $\delta e^{i\phi}$
($\phi:(3\pi)/2\to \pi/2$).
Then we see from the above discussion in addition to
Propositions \ref{Prop:negative1} and \ref{Prop:negative2}
that $\Rep(L'/L)(s,\chi)<0$ on $s\in\mathcal{C}$.
Since $(L'/L)(s,\chi)$ has no poles inside $\mathcal{C}$,
the argument principle says that $L'(s,\chi)$ has no zeros
inside $\mathcal{C}$.
Since $\delta\in(0,\delta_0]$ is arbitrary,
$L'(s,\chi)$ has no zeros on $\{s\in\C:-1\leq\Rep(s)\leq 0,s\neq 0\}$.
Combining the fact that $s=0$ is a simple zero of $L(s,\chi)$,
we obtain the first claim of Theorem \ref{Thm4}.

Next we consider the case $\chi(-1)=-1$.
We take the contour $\mathcal{C}$ determined by the rectangle
with vertices at $-2\pm 1000i$, $\pm 1000i$.
Then we have already shown that $\Rep(L'/L)(s,\chi)<0$ holds
on $s\in\mathcal{C}\setminus[-1000i,1000i]$.
Let $t\in[-1000,1000]$.
Then in the same manner as the case $\chi(-1)=1$,
(\ref{eq:Hadamard2}) gives
\begin{equation}\label{eq17}
 \Rep\frac{L'}{L}(it,\chi)\leq
-\frac{1}{2}\frac{\Gamma'}{\Gamma}\left(\frac{1}{2}\right)
-\frac{1}{2}\log\frac{q}{\pi}.
\end{equation}
Since $(\Gamma'/\Gamma)(1/2)=-2\log 2-c_E$,
(\ref{eq17}) is negative provided
$q>4\pi e^{c_E}=22.38\ldots$.
Thus $\Rep(L'/L)(s,\chi)<0$ holds on $s\in\mathcal{C}$
if $q\geq 23$.
Applying the argument principle and taking the trivial zero
$s=-1$ of $L(s,\chi)$ into account, we see that $L'(s,\chi)$ has
a unique zero inside $\mathcal{C}$.
This completes the proof.
\end{proof}
Finally in this section we briefly mention the case when $\chi$ is quadratic.
In this case we have the following assertions, which give more detailed
information than Theorems \ref{Thm2} and \ref{Thm4}:
\begin{Proposition}\label{Prop:quadratic1}
 Let $\chi$ be a primitive quadratic Dirichlet character modulo $q>1$.
Then for each $j\in\Z_{\geq 1}$
the zero of $L'(s,\chi)$ in $-2j-\kappa-1<\Rep(s)<-2j-\kappa+1$
lies in the interval $(-2j-\kappa,-2j-\kappa+1)$.
\end{Proposition}
\begin{Proposition}\label{Prop:quadratic2}
 Let $\chi$ be an odd primitive quadratic Dirichlet character modulo
$q\geq 23$.
Then the zero of $L'(s,\chi)$ on $-2\leq \Rep(s)\leq 0$
lies in the interval $(-1,0)$.
\end{Proposition}
\begin{proof}[Proof of Propositions \ref{Prop:quadratic1} and \ref{Prop:quadratic2}]
Let $\chi$ be a primitive quadratic Dirichlet character modulo $q>1$.
Then $(L'/L)(s,\chi)$ is real
if $s\in\{s\in\R:L(s,\chi)\neq 0\}$.
By the functional equation (\ref{eq:fteq}) we have
\begin{equation}\label{eq:plusinfty}
 \lim_{s\downarrow -2j-\kappa}\frac{L'}{L}(s,\chi)=+\infty
\end{equation}
for each $j\in\Z_{\geq 0}$.
Combining this with Proposition \ref{Prop:negative2},
we obtain Proposition \ref{Prop:quadratic1} by the intermediate value theorem.
In the same manner,
(\ref{eq17}) and (\ref{eq:plusinfty}) give Proposition \ref{Prop:quadratic2}.
\end{proof}
\section{Zeros of $L'(s,\chi)$ in $\Rep(s)>0$}
In this section we show Theorems \ref{Thm5} and \ref{Thm6}.

For short we write the functional equation for $L(s,\chi)$
as $L(s,\chi)=F(s,\chi)L(1-s,\overline{\chi})$,
where
\[
 F(s,\chi)=\varepsilon(\chi)2^s\pi^{s-1}q^{\frac{1}{2}-s}
\sin\left(\frac{\pi(s+\kappa)}{2}\right)\Gamma(1-s).
\]
Here $\varepsilon(\chi)$ is a constant depending on $\chi$,
which satisfies $|\varepsilon(\chi)|=1$.
We also define $G(s,\chi)$ by
\[
 G(s,\chi)=-\frac{m^s}{\chi(m)\log m}L'(s,\chi).
\]
First of all we show
\begin{Lemma}\label{Lem:Gest}
 For $s=\sigma+it$ with $\sigma\geq 2$ and $t\in\R$ we have
\[
 |G(s,\chi)-1|
\leq 2\left(1+\frac{8m}{\sigma}\right)\exp\left(-\frac{\sigma}{2m}\right).
\]
\end{Lemma}
\begin{proof}
 By the Dirichlet series expression for $L(s,\chi)$
we find
\[
 G(s,\chi)=1+\frac{m^s}{\chi(m) \log m}
\sum_{n=m+1}^{\infty}\frac{\chi(n)\log n}{n^s}.
\]
Thus we have
\begin{equation}\label{eq:Gest1}
 |G(s,\chi)-1|\leq\frac{m^{\sigma}}{\log m}
\sum_{n=m+1}^{\infty}\frac{\log n}{n^{\sigma}}.
\end{equation}
We divide the sum into $n=m+1$ and $n\geq m+2$.
The sum over $n\geq m+2$ is estimated as follows:
\begin{align*}
 \sum_{n=m+2}^{\infty}\frac{\log n}{n^{\sigma}}
&\leq\int_{m+1}^{\infty}\frac{\log u}{u^{\sigma}}du\\
&=\frac{(m+1)^{1-\sigma}\log(m+1)}{\sigma-1}
+\frac{(m+1)^{1-\sigma}}{(\sigma-1)^2}\\
&\leq\frac{2(m+1)^{1-\sigma}\log(m+1)}{\sigma-1}.
\end{align*}
Inserting this into (\ref{eq:Gest1}), we have
\begin{equation}\label{eq:Gest2}
\begin{aligned}
 |G(s,\chi)-1|
&\leq\frac{\log(m+1)}{\log m}\left(\frac{m}{m+1}\right)^{\sigma}
\left(1+2\frac{m+1}{\sigma-1}\right)\\
&\leq2\left(1+\frac{8m}{\sigma}\right)\left(\frac{m}{m+1}\right)^{\sigma}.
\end{aligned}
\end{equation}
Since $\log(1+x)\geq x/2$ on $x\in[0,1]$,
we find
\[
 \left(\frac{m}{m+1}\right)^{\sigma}
=\exp\left(-\sigma\log\left(1+\frac{1}{m}\right)\right)
\leq\exp\left(-\frac{\sigma}{2m}\right).
\]
Applying this to (\ref{eq:Gest2}), we obtain the result.
 \end{proof}
By Lemma \ref{Lem:Gest} we have
\begin{equation}\label{eq:Gest3}
 |G(s,\chi)-1|\leq 4\exp\left(-\frac{\sigma}{2m}\right)
\end{equation}
for $\sigma\geq 8m$.
In particular, the function $G(s,\chi)$ has no zeros on $\sigma\geq 8m$.

Let $b_{\kappa}\in\{1+\kappa,3+\kappa\}$, $T\geq 2$
and $U\geq 10m$.
We apply the Littlewood lemma (see \cite[\S3.8]{Ti}) to $G(s,\chi)$
on the rectangle with vertices at $-b_{\kappa}\pm iT$
and $U\pm iT$.
Taking the imaginary part, we have
\begin{equation}\label{eq:Littlewood}
\begin{aligned}
&2\pi
\sum_{\begin{subarray}{c}
       \rho'=\beta'+i\gamma'\\
       \beta'>-b_{\kappa}, -T\leq \gamma'\leq T
      \end{subarray}}
(\beta'+b_{\kappa})\\
&=\int_{-T}^T\log|G(-b_{\kappa}+it,\chi)|dt
-\int_{-T}^T\log|G(U+it,\chi)|dt\\
&+\int_{-b_{\kappa}}^U\arg G(\sigma+iT,\chi)d\sigma
-\int_{-b_{\kappa}}^U\arg G(\sigma-iT,\chi)d\sigma.
\end{aligned}
\end{equation}
Here we determine the branch of $\log G(s,\chi)$
such that it tends to $0$ as $\sigma\to\infty$
and it is holomorphic in
$\C\setminus\{\rho'+\lambda:L'(\rho',\chi)=0,~\lambda\leq 0\}$.
When there are zeros of $L'(s,\chi)$ on $\Imp(s)=\pm T$,
we determine
$\arg G(\sigma\pm iT,\chi)=\lim_{\varepsilon\downarrow 0}
\arg G(\sigma\pm i(T+\varepsilon),\chi)$.
Thanks to (\ref{eq:Gest3}), the second integral on (\ref{eq:Littlewood})
tends to $0$ as $U\to\infty$.
We also note that Theorems \ref{Thm1}, \ref{Thm2} and \ref{Thm4}
give
$\#\{\rho'=\beta'+i\gamma':
L'(\rho',\chi)=0,~-b_{\kappa}<\beta'\leq 0\}\ll 1$,
where the implied constant is absolute.
Combining these, we obtain
\begin{equation}\label{eq:Littlewood2}
 2\pi
\sum_{\begin{subarray}{c}
       \rho'=\beta'+i\gamma'\\
       \beta'>0, -T\leq \gamma'\leq T
      \end{subarray}}
(\beta'+b_{\kappa})
=I_1+I_2^{+}-I_2^{-}+O(1),
\end{equation}
where $I_1=I_1(b_{\kappa},\chi,T)$ and
$I_2^{\pm}=I_2^{\pm}(b_{\kappa},\chi,T)$
are given by
\begin{align*}
 I_1&=\int_{-T}^{T}\log|G(-b_{\kappa}+it,\chi)|dt,\\
 I_2^{\pm}&=\int_{-b_{\kappa}}^{\infty}\arg G(\sigma\pm iT,\chi)d\sigma.
\end{align*}

We deal with $I_1$.
By the definition of $G(s,\chi)$ we have
\begin{equation}\label{eq:I1-1}
 I_1=-2T(b_{\kappa}\log m+\log\log m)
+\int_{-T}^T\log|L'(-b_{\kappa}+it,\chi)|dt.
\end{equation}
We treat the last integral.
We note that $\overline{L'(\overline{s},\overline{\chi})}=L'(s,\chi)$
gives
\begin{equation}\label{eq:I1-1-5}
 \int_{-T}^0\log|L'(-b_{\kappa}+it,\chi)|dt
=\int_0^T\log|L'(-b_{\kappa}+it,\overline{\chi})|dt.
\end{equation}
So it suffices to consider the integral over $t\in[0,T]$.
We take an absolute constant $t_0\geq 2$ and we will
determine it afterward.
See the discussion around (\ref{eq:I1-7-3}) for a choice of $t_0$.
We suppose $T\geq t_0$.
We divide the interval $[0,T]$ into $[0,t_0]$
and $(t_0,T]$.
We treat the integral over $[0,t_0]$.
We have
\[
 \log|L'(-b_{\kappa}+it,\chi)|
=\log|L(-b_{\kappa}+it,\chi)|
+\log\left|
\frac{L'}{L}(-b_{\kappa}+it,\chi)
\right|.
\]
By the functional equation, the first term on the right
is $(\frac{1}{2}+b_{\kappa})\log q+O(1)$ uniformly on $t\in[0,t_0]$.
We see from the functional equation together with Proposition \ref{Prop:negative2}
 that the second term is $O(\log\log q)$ on $t\in[0,t_0]$.
In consequence we obtain
\begin{equation}\label{eq:I1-2}
 \int_0^{t_0}\log|L'(-b_{\kappa}+it,\chi)|dt\ll \log q.
\end{equation}
Next we deal with the integral over $(t_0,T]$.
By the functional equation we have
$L'(s,\chi)=F'(s,\chi)L(1-s,\overline{\chi})-F(s,\chi)L'(1-s,\overline{\chi})$,
so that
\begin{equation}\label{eq:I1-3}
 \begin{aligned}
&\int_{t_0}^T\log|L'(-b_{\kappa}+it,\chi)|dt\\
&=\int_{t_0}^T\log|F(-b_{\kappa}+it,\chi)|dt
+\int_{t_0}^T\log\left|\frac{F'}{F}(-b_{\kappa}+it,\chi)\right|dt\\
&+\int_{t_0}^T\log|L(1+b_{\kappa}-it,\overline{\chi})|dt\\
&+\int_{t_0}^T
\log\left|
1-\frac{1}{(F'/F)(-b_{\kappa}+it,\chi)}
\frac{L'}{L}(1+b_{\kappa}-it,\overline{\chi})
\right|
dt.
 \end{aligned}
\end{equation}
By Stirling's formula we have
\[
 \log|F(-b_{\kappa}+it,\chi)|
=\left(\frac{1}{2}+b_{\kappa}\right)\log\frac{qt}{2\pi}+
O\left(\frac{1}{t}\right).
\]
Consequently,
\begin{equation}\label{eq:I1-4}
 \int_{t_0}^T\log|F(-b_{\kappa}+it,\chi)|dt
=\left(\frac{1}{2}+b_{\kappa}\right)
\left(
T\log\frac{qT}{2\pi}-T
\right)
+O(\log(qT)).
\end{equation}
In a similar manner, Stirling's formula for $(\Gamma'/\Gamma)(z)$
gives
\[
 \frac{F'}{F}(-b_{\kappa}+it,\chi)=-\log\frac{qt}{2\pi}
+O\left(\frac{1}{t}\right).
\]
Thus we have
\begin{equation}\label{eq:I1-5}
 \int_{t_0}^T
\log\left|
\frac{F'}{F}(-b_{\kappa}+it,\chi)
\right|
dt
=\int_{t_0}^T\log\log\frac{qt}{2\pi}dt
+O\left(
\int_{t_0}^T\frac{dt}{t\log(qt)}
\right).
\end{equation}
Integrating by parts, we see that the first integral on the right
turns to
\begin{align*}
 \int_{t_0}^T\log\log\frac{qt}{2\pi}dt
&=T\log\log\frac{qT}{2\pi}-\frac{2\pi}{q}\li\left(\frac{qT}{2\pi}\right)
+O(\log\log q).
\end{align*}
We easily see that the last term on (\ref{eq:I1-5}) is $O(\log\log(qT))$.
Combining these, we obtain
\begin{equation}\label{eq:I1-6}
  \int_{t_0}^T
\log\left|
\frac{F'}{F}(-b_{\kappa}+it,\chi)
\right|
dt
=T\log\log\frac{qT}{2\pi}-\frac{2\pi}{q}\li\left(\frac{qT}{2\pi}\right)
+O(\log\log(qT)).
\end{equation}
We see from the Dirichlet series expression for $\log L(s,\overline{\chi})$
that
\begin{equation}\label{eq:I1-7}
 \int_{t_0}^T\log|L(1+b_{\kappa}-it,\overline{\chi})|dt\ll 1.
\end{equation}
Next we treat the last term on (\ref{eq:I1-3}).
For this purpose we determine $t_0$ and we estimate
the integrand.
By Stirling's formula,
$(F'/F)(s,\chi)=-\log(q|1-s|)+O(1)$ holds
for $\sigma\leq -1$ and $t\geq 2$,
where the implied constant is absolute.
So we can choose $t_0$, which does not depend on any parameters,
such that
\begin{equation}\label{eq:I1-7-3}
\left|\frac{F'}{F}(s,\chi)\right|\geq 10
\end{equation}
holds for $\sigma\leq -1$ and $t\geq t_0$.
On the other hand,
by the Dirichlet series expression for $(L'/L)(s,\chi)$ we have
\begin{align*}
 \left|\frac{L'}{L}(1-s,\overline{\chi})\right|
&\leq\sum_{n=2}^{\infty}\frac{\log n}{n^{1-\sigma}}
\leq\frac{\log 2}{2^{1-\sigma}}+\int_2^{\infty}\frac{\log u}{u^{1-\sigma}}du\\
&\leq\left(1+\frac{3}{2}\log 2\right)2^{\sigma}
\end{align*}
for $\sigma\leq -1$ and $t\geq t_0$.
Thus we find that
 \begin{equation}\label{eq:I1-7-5}
 \left|
\frac{1}{(F'/F)(s,\chi)}\frac{L'}{L}(1-s,\overline{\chi})
\right|
\leq 2^{\sigma}
\end{equation}
holds for $\sigma\leq -1$ and $t\geq t_0$.
Therefore we can determine the branch of
\begin{equation}\label{eq:I1-8}
 \log\left(
   1-\frac{1}{(F'/F)(s,\chi)}\frac{L'}{L}(1-s,\overline{\chi})
\right)
\end{equation}
such that it is holomorphic in a region including
$\{\sigma+it:\sigma\leq -1, t\geq t_0\}$
and it tends to $0$ as $\sigma\to -\infty$.
We apply Cauchy's theorem to (\ref{eq:I1-8}) on the triangle
joining $-b_{\kappa}+it_0$, $-b_{\kappa}+iT$ and $-T+iT$.
The inequality (\ref{eq:I1-7-5}) says that (\ref{eq:I1-8}) is $O(2^{\sigma})$
on the triangle.
This gives 
\begin{equation}\label{eq:I1-9}
 \int_{t_0}^T
\log\left|
1-\frac{1}{(F'/F)(-b_{\kappa}+it,\chi)}
\frac{L'}{L}(1+b_{\kappa}-it,\overline{\chi})
\right|
dt\ll 1.
\end{equation}
We insert (\ref{eq:I1-4}), (\ref{eq:I1-6}), (\ref{eq:I1-7})
and (\ref{eq:I1-9}) into (\ref{eq:I1-3}).
Combining this with (\ref{eq:I1-2}), we obtain
 \begin{align*}
 &\int_0^T\log|L'(-b_{\kappa}+it,\chi)|dt\\
&=\left(\frac{1}{2}+b_{\kappa}\right)\left(T\log\frac{qT}{2\pi}-T\right)
+T\log\log\frac{qT}{2\pi}-\frac{2\pi}{q}\li\left(\frac{qT}{2\pi}\right)
+O(\log(qT)).
 \end{align*}
Thanks to (\ref{eq:I1-1-5}), a similar formula holds for the
integral over $[-T,0]$.
Applying these to (\ref{eq:I1-1}), we conclude
\begin{equation}\label{eq:I1}
 \begin{aligned}
  I_1&=
2\left(\frac{1}{2}+b_{\kappa}\right)\left(T\log\frac{qT}{2\pi}-T\right)
-2T(b_{\kappa}\log m+\log\log m)\\
&+2T\log\log\frac{qT}{2\pi}-\frac{4\pi}{q}\li\left(\frac{qT}{2\pi}\right)
+O(\log(qT)).
 \end{aligned}
\end{equation}
This remains valid for $2\leq T<t_0$.
In fact, when $2\leq T<t_0$, we find $I_1=O(\log q)$ in the same
manner as (\ref{eq:I1-2}), which implies (\ref{eq:I1}).

Next we deal with $I_2^{\pm}$.
For this purpose we will give the following
bounds for $\arg G(\sigma\pm iT,\chi)$:
\begin{Proposition}\label{Prop:argG}
 For $T\geq 2$ we have
\begin{equation}\label{eq:argG}
 \arg G(\sigma\pm iT,\chi)\ll
\begin{cases}
 \exp(-\sigma/(2m))& \text{if $10m\leq \sigma$,}\\
 m/\sigma           & \text{if $3\leq\sigma\leq 10m$,}\\
 m^{1/2}\log(qT)   & \text{if $-5\leq\sigma\leq 3$},
\end{cases}
\end{equation}
where the implied constant is absolute.
\end{Proposition}
In order to show this, we collect consequences of well-known facts.
First of all we recall estimates for $G(s,\chi)$.
\begin{Lemma}\label{Lem:estimateG}
For $s=\sigma+it$ with $-10\leq\sigma\leq 3$ and $t\in\R$
we have
\[
 G(s,\chi)\ll (q\tau)^{20},
\]
where $\tau:=|t|+2$ and the implied constant is absolute.
\end{Lemma}
\begin{proof}
Cauchy's integral formula gives
\begin{equation}\label{eq:Cauchy}
 L'(s,\chi)=\frac{1}{2\pi i}
\int_{|w-s|=1}
\frac{L(w,\chi)}{(w-s)^2}dw.
\end{equation}
According to \cite[Corollary 10.10 and Lemma 10.15]{MV},
the inequality $L(s,\chi)\ll (q\tau)^{15}$ holds
for $-11\leq\sigma\leq 4$ and $t\in\R$.
Inserting this into (\ref{eq:Cauchy})
and using $m\ll\log q$,
we reach the result.
\end{proof}
Next we recall the following formula:
\begin{Lemma}\label{Lem:Jensen}
 For $a>0$ and $b>0$ we have
\[
 \frac{1}{2\pi}\int_0^{2\pi}\log|a+b\cos\theta|d\theta
=\begin{cases}
\log\frac{a+\sqrt{a^2-b^2}}{2} & \text{if $a>b$},\\
\log(b/2) & \text{if $a\leq b$}.
\end{cases}
\]
\end{Lemma}
\begin{proof}
We calculate the left-hand side as
\begin{equation}\label{eq:logcos}
\begin{aligned}
&\frac{1}{2\pi}\int_0^{2\pi}\log|a+b\cos\theta|d\theta\\
&=\frac{1}{2\pi}\int_0^{2\pi}
\log\left|a+b\frac{e^{i\theta}+e^{-i\theta}}{2}\right|d\theta\\
&=\log\left(\frac{b}{2}\right)
+\frac{1}{2\pi}\int_0^{2\pi}
\log\left|e^{2i\theta}+\frac{2a}{b}e^{i\theta}+1\right|d\theta.
\end{aligned}
\end{equation}
We put $\alpha_{\pm}=-\frac{a}{b}\pm\sqrt{(\frac{a}{b})^2-1}$,
which are solutions of $X^2+\frac{2a}{b}X+1=0$.
By Jensen's theorem (see \cite[\S 3.61]{Ti}), (\ref{eq:logcos})
turns to
\[
=\log\left(\frac{b}{2}\right)+\log^{+}|\alpha_{+}|+\log^{+}|\alpha_{-}|,
\]
where $\log^{+}x=\max\{\log x,0\}$.
We can easily check that $|\alpha_{+}|<1$ and $|\alpha_{-}|>1$ when $a>b$
and that $|\alpha_{\pm}|=1$ when $a\leq b$.
This completes the proof.
\end{proof}
Now we are ready to prove Proposition \ref{Prop:argG}.
In the proof below $c_1,c_2,\ldots$ are positive constants independent
of any parameters.
\begin{proof}[Proof of Proposition \ref{Prop:argG}]
We see from $\overline{G(\overline{s},\overline{\chi})}=G(s,\chi)$
that $\arg G(\sigma-iT,\chi)=-\arg G(\sigma+iT,\overline{\chi})$.
Thus it suffices to show (\ref{eq:argG}) for $\arg G(\sigma+iT,\chi)$
only.
We concentrate on $\arg G(\sigma+iT,\chi)$ below.
When $\sigma\geq 10m$, (\ref{eq:argG}) is an immediate consequence
of Lemma \ref{Lem:Gest} or (\ref{eq:Gest3}).

Let $\sigma\in[-10,10m]$.
We put
$h:=\#\{x\in[\sigma,10m]:\Rep G(x+iT,\chi)=0\}$.
Then we see that $|\arg G(\sigma+iT,\chi)|\leq(h+\frac{3}{2})\pi$.
In order to estimate $h$, we put
\[
 H(z,\chi):=\frac{1}{2}(G(z+iT,\chi)+G(z-iT,\overline{\chi})).
\]
For $r>0$ we denote by $n(r)$ the number of zeros of $H(z,\chi)$
on $|z-11m|\leq r$, counted with multiplicity.
Since $H(z,\chi)=\Rep G(z+iT,\chi)$ for $z\in\R$,
we see that $h\leq n(R)$, where
\[
 R:=11m-\sigma.
\]
We see from the above discussion that
\begin{equation}\label{eq:argG2}
 \arg G(\sigma+iT,\chi)\ll n(R).
\end{equation}

Below we estimate $n(R)$. We take $R_0>0$.
Then by Jensen's theorem we have
\[
 \int_0^{R+R_0}\frac{n(r)}{r}dr
=\frac{1}{2\pi}\int_0^{2\pi}\log|H(11m+(R+R_0)e^{i\theta},\chi)|d\theta
-\log|H(11m,\chi)|.
\]
Since $n(r)$ is nonnegative and monotonically increasing,
the left-hand side is bounded below as
\[
 \int_0^{R+R_0}\frac{n(r)}{r}dr\geq\int_{R}^{R+R_0}\frac{n(r)}{r}dr
\geq n(R)\log\left(1+\frac{R_0}{R}\right).
\]
Combining this with $\log|H(11m,\chi)|=O(1)$,
which follows from (\ref{eq:Gest3}),
we have
\begin{equation}\label{eq:estnR}
 n(R)\leq
\frac{1}{\log(1+\frac{R_0}{R})}
\left(
\frac{1}{2\pi}\int_0^{2\pi}\log|H(11m+(R+R_0)e^{i\theta},\chi)|d\theta
+c_1
\right).
\end{equation}

First of all we consider the case $3\leq\sigma\leq 10m$.
In this case we restrict $R_0$ by
\begin{equation}\label{eq:restriction}
 0<R_0\leq \sigma-2.
\end{equation}
Then we note $11m-(R+R_0)\geq 2$.
We see from Lemma \ref{Lem:Gest} that
\begin{align*}
&\frac{1}{2\pi}\int_0^{2\pi}\log|H(11m+(R+R_0)e^{i\theta},\chi)|d\theta\\
&\leq
\frac{1}{2\pi}\int_0^{2\pi}\log\frac{m}{11m+(R+R_0)\cos\theta}d\theta+c_2.
\end{align*}
By Lemma \ref{Lem:Jensen}, this is
\[
 \leq\log m-\log\left(\frac{11m}{2}\right)+c_2\leq c_3.
\]
We also note that the restriction (\ref{eq:restriction}) implies
$0<R_0/R\leq c_4$, so that $\log(1+\frac{R_0}{R})\gg R_0/R$.
Combining these, we obtain
\[
 n(R)\ll\frac{R}{R_0}.
\]
Taking $R_0=\sigma-2$, we obtain $n(R)\ll m/\sigma$.
This together with (\ref{eq:argG2}) completes the proof
when $3\leq\sigma\leq 10m$.

Finally we deal with the case $-5\leq\sigma\leq 3$.
In this case we choose $R_0=5$.
In order to estimate the integral on (\ref{eq:estnR}),
we divide $[0,2\pi]=\mathcal{I}_1\cup \mathcal{I}_2$,
where
\begin{align*}
 \mathcal{I}_1&:=\{\theta\in[0,2\pi]:11m+(R+5)\cos\theta\geq 2\},\\
 \mathcal{I}_2&:=\{\theta\in[0,2\pi]:11m+(R+5)\cos\theta<2\}.
\end{align*}
We take $\theta_0\in(0,\pi/2)$ such that
\[
 \cos\theta_0=\frac{11m-2}{R+5}.
\]
Then we have $\mathcal{I}_1=[0,\pi-\theta_0]\cup[\pi+\theta_0,2\pi]$
and $\mathcal{I}_2=(\pi-\theta_0,\pi+\theta_0)$.
Since $\cos\theta_0=1+O(1/m)$ and $\cos\theta_0=1-2\sin^2(\theta_0/2)$,
we see that
\begin{equation}\label{eq:esttheta}
 \theta_0=O(m^{-1/2}).
\end{equation}
We deal with the integral over $\mathcal{I}_1$.
By Lemma \ref{Lem:Gest} we have
\begin{equation}\label{eq:estintegral}
\begin{aligned}
&\frac{1}{2\pi}\int_{\mathcal{I}_1}\log|H(11m+(R+5)e^{i\theta},\chi)|d\theta\\
&\leq\log m-\frac{1}{2\pi}\int_{\mathcal{I}_1}\log|11m+(R+5)\cos\theta|
d\theta+c_5.
\end{aligned}
\end{equation}
We see from Lemma \ref{Lem:Jensen} together with $R+5\geq 11m$
that
\begin{align*}
&\frac{1}{2\pi}\int_{\mathcal{I}_1}
\log|11m+(R+5)\cos\theta|d\theta\\
&=\log\frac{R+5}{2}
-\frac{1}{2\pi}\int_{\pi-\theta_0}^{\pi+\theta_0}
\log|11m+(R+5)\cos\theta|d\theta\\
&\geq\log m-\frac{\theta_0}{\pi}\log(30m).
\end{align*}
Inserting this into (\ref{eq:estintegral}) and using (\ref{eq:esttheta}),
we obtain
\[
 \frac{1}{2\pi}\int_{\mathcal{I}_1}\log|H(11m+(R+5)e^{i\theta},\chi)|d\theta
\leq c_6.
\]
Next we treat the integral over $\mathcal{I}_2$.
By Lemma \ref{Lem:estimateG},
\[
 H(11m+(R+5)e^{i\theta},\chi)\ll (qT')^{20}
\]
holds on $\theta\in\mathcal{I}_2$, where $T':=\max\{T,m\}$.
This together with (\ref{eq:esttheta}) yields
\[
 \frac{1}{2\pi}\int_{\mathcal{I}_2}\log|H(11m+(R+5)e^{i\theta},\chi)|d\theta
\leq c_7 m^{-1/2}\log(qT').
\]
Inserting this and $\log(1+\frac{5}{R})\gg 1/R\gg 1/m$ into (\ref{eq:estnR}),
we obtain
\begin{equation}\label{eq:estnR2}
n(R)\ll m(m^{-1/2}\log(qT')+1)\ll m^{1/2}\log(qT')\ll m^{1/2}\log(qT).
\end{equation}
Here in the second inequality we used $m\ll\log q$.
In the last inequality
we also used $\log(qT')\ll\log(q\log q)\ll\log q\ll\log(qT)$ when $T\leq m$.
Applying (\ref{eq:estnR2}) to (\ref{eq:argG2}), we reach
the result when $-5\leq\sigma\leq 3$.

The proof of Proposition \ref{Prop:argG} is completed.
\end{proof}
\begin{proof}[Proof of Theorem \ref{Thm5}]
 Subtracting (\ref{eq:Littlewood2}) with $b_{\kappa}=1+\kappa$
from that with $b_{\kappa}=3+\kappa$, we have
\begin{align*}
4\pi N_1(T,\chi)
&=(I_1(3+\kappa,\chi,T)-I_1(1+\kappa,\chi,T))\\
&+(I_2^{+}(3+\kappa,\chi,T)-I_2^{+}(1+\kappa,\chi,T))\\
&-(I_2^{-}(3+\kappa,\chi,T)-I_2^{-}(1+\kappa,\chi,T))
+O(1).
\end{align*}
By (\ref{eq:I1}) we have
\[
I_1(3+\kappa,\chi,T)-I_1(1+\kappa,\chi,T)
=4T\log\frac{qT}{2\pi m}-4T+O(\log(qT)).
\]
On the other hand, Proposition \ref{Prop:argG} gives
 \begin{align*}
  I_2^{\pm}(3+\kappa,\chi,T)-I_2^{\pm}(1+\kappa,\chi,T)
&=\int_{-3-\kappa}^{-1-\kappa}\arg G(\sigma\pm iT,\chi)d\sigma\\
&\ll m^{1/2}\log(qT).
 \end{align*}
Combining these, we obtain the result.
\end{proof}
\begin{proof}[Proof of Theorem \ref{Thm6}]
 We start with (\ref{eq:Littlewood2}).
We estimate $I_2^{\pm}=I_2^{\pm}(b_{\kappa},\chi,T)$.
By Proposition \ref{Prop:argG} we have
\begin{equation}\label{eq:I2}
 I_2^{\pm}\ll m\log m+m^{1/2}\log(qT)\ll m^{1/2}\log(qT).
\end{equation}
Here in the last inequality we used $m\ll\log q$.
We also note that
\begin{align*}
&2\pi
\sum_{\begin{subarray}{c}
       \rho'=\beta'+i\gamma'\\
       \beta'>0, -T\leq \gamma'\leq T
      \end{subarray}}
(\beta'+b_{\kappa})\\
&=2\pi
\sum_{\begin{subarray}{c}
       \rho'=\beta'+i\gamma'\\
       \beta'>0, -T\leq \gamma'\leq T
      \end{subarray}}
\left(\beta'-\frac{1}{2}\right)
+2\pi\left(b_{\kappa}+\frac{1}{2}\right)N_1(T,\chi).
\end{align*}
Applying Theorem \ref{Thm5}, (\ref{eq:I1}) and (\ref{eq:I2}),
we complete the proof.
\end{proof}
\section{Analogues of Speiser's theorem}
In this section we show Theorems \ref{Thm8}--\ref{Thm7}.
First of all we investigate the sign of $\Rep(L'/L)(s,\chi)$
on $\Rep(s)=1/2$.
For convenience we put
\[
 \mathcal{T}=\mathcal{T}_{\chi}:=\{t\in\R:L(\tfrac{1}{2}+it,\chi)\neq 0\}.
\]
\begin{Lemma}\label{Lem:center}
Let $\chi$ be a primitive Dirichlet character modulo $q>1$.
Then for $t\in\mathcal{T}$
\begin{equation}\label{eq:negative}
 \Rep\frac{L'}{L}\left(\frac{1}{2}+it,\chi\right)<0
\end{equation}
holds if one of the following conditions holds:
\begin{enumerate}
 \item $\kappa=0$ and $q\geq 216$.
 \item $\kappa=0$ and $|t|\geq 2$.
 \item $\kappa=1$ and $q\geq 10$.
 \item $\kappa=1$ and $|t|\geq 3$.
\end{enumerate}
\end{Lemma}
\begin{proof}
 We begin with (\ref{eq:Hadamard3}).
Since $L(s,\chi)=F(s,\chi)L(1-s,\overline{\chi})$ and 
$\overline{L(\overline{s},\overline{\chi})}=L(s,\chi)$,
each zero of $L(s,\chi)$ in $\Rep(s)>1/2$ can be written
by $1-\overline{\rho}$ uniquely, where $\rho=\beta+i\gamma$ is a zero
of $L(s,\chi)$ in $0<\beta<1/2$.
Furthermore, routine calculation gives
\[
 \frac{\sigma-\beta}{|s-\rho|^2}+\frac{\sigma-(1-\beta)}{|s-(1-\overline{\rho})|^2}
=(2\sigma-1)
\frac{(\sigma-\frac{1}{2})^2-(\beta-\frac{1}{2})^2+(t-\gamma)^2}
{|s-\rho|^2|s-1+\overline{\rho}|^2}.
\]
Applying these to (\ref{eq:Hadamard3}), for $s=\sigma+it$
with $L(s,\chi)\neq 0$ we find
\begin{equation}\label{eq:Hadamard4}
 \Rep\frac{L'}{L}(s,\chi)
=-\frac{1}{2}\log\frac{q}{\pi}
-\frac{1}{2}\Rep\frac{\Gamma'}{\Gamma}
\left(
\frac{s+\kappa}{2}
\right)
+\left(\sigma-\frac{1}{2}\right)J(s,\chi),
\end{equation}
where
\[
 J(s,\chi)=\sum_{\beta=\frac{1}{2}}\frac{1}{|s-\rho|^2}
+2\sum_{\beta<1/2}
\frac{(\sigma-\frac{1}{2})^2-(\beta-\frac{1}{2})^2+(t-\gamma)^2}
{|s-\rho|^2|s-1+\overline{\rho}|^2}.
\]
Thus, for $t\in\mathcal{T}$ we have
\begin{equation}\label{eq:center1}
 \Rep\frac{L'}{L}\left(\frac{1}{2}+it,\chi\right)
=-\frac{1}{2}\log\frac{q}{\pi}
-\frac{1}{2}\Rep\frac{\Gamma'}{\Gamma}
\left(
\frac{1}{4}+\frac{\kappa}{2}+\frac{it}{2}
\right).
\end{equation}
We note that the right-hand side is an even function with respect to $t$.
Therefore we concentrate on $t\geq 0$ below.
Let $t_1\in[0,\infty)$.
Since the right-hand side of (\ref{eq:center1})
is monotonically decreasing on $t\geq 0$ thanks to (\ref{eq:Gammaineq}),
for $t\in\mathcal{T}\cap[t_1,\infty)$
we have
\begin{equation}\label{eq:center2}
  \Rep\frac{L'}{L}\left(\frac{1}{2}+it,\chi\right)
\leq-\frac{1}{2}\log\frac{q}{\pi}
-\frac{1}{2}\Rep\frac{\Gamma'}{\Gamma}
\left(
\frac{1}{4}+\frac{\kappa}{2}+\frac{it_1}{2}
\right).
\end{equation}

We take $t_1=0$. Then (\ref{eq:negative}) holds for
$t\in\mathcal{T}$, provided
\begin{equation}\label{eq:qcond}
 q>\pi\exp\left(-\frac{\Gamma'}{\Gamma}\left(\frac{1}{4}+\frac{\kappa}{2}\right)\right).
\end{equation}
By \cite[(C.15) and (C.16)]{MV}, the right-hand side of (\ref{eq:qcond})
equals
\[
 =8\pi\exp\left(c_E+(-1)^{\kappa}\frac{\pi}{2}\right)
=\begin{cases}
  215.3\ldots & \text{if $\kappa=0$,}\\
  9.3\ldots   & \text{if $\kappa=1$.}
 \end{cases}
\]
Thus (\ref{eq:negative}) holds if the condition (1) or (3) is
satisfied.

We go back to (\ref{eq:center2}) and consider the case $\kappa=0$.
In this case $q\geq 5$ holds.
We have
\[
 \log\frac{5}{\pi}>0.46.
\]
On the other hand, by numerical computation together with
(\ref{eq:Gammaineq}) and \cite[(C.15)]{MV} we find
\begin{align*}
 \Rep\frac{\Gamma'}{\Gamma}\left(\frac{1}{4}+i\right)
&=-c_E-\frac{\pi}{2}-3\log 2
+\sum_{n=0}^{\infty}\frac{1}{(n+\frac{1}{4})\{(n+\frac{1}{4})^2+1\}}\\
&>-0.04.
\end{align*}
Here in the last inequality we discarded the sum over $n>5$
and carried out a numerical calculation.
Combining these and (\ref{eq:center2}), we see that (\ref{eq:negative})
holds if the condition (2) is satisfied.

Finally we treat the case $\kappa=1$.
We note that $\kappa=1$ implies $q\geq 3$.
In a similar manner as the case $\kappa=0$ we find
\[
\log\frac{3}{\pi}>-0.05\text{ and }
\Rep\frac{\Gamma'}{\Gamma}\left(\frac{3}{4}+\frac{3i}{2}\right)
>0.37.
\]
This together with (\ref{eq:center2}) says that
(\ref{eq:negative}) holds under the condition (4).
\end{proof}
By Lemma \ref{Lem:center} we immediately see
\begin{Corollary}
 Let $\chi$ be a primitive Dirichlet character modulo $q$.
Suppose that $\kappa=0$ and $q\geq 216$, or that
$\kappa=1$ and $q\geq 10$.
Let $t\in\R$.
If $L'(\frac{1}{2}+it,\chi)=0$,
then $s=\frac{1}{2}+it$ is a multiple zero of $L(s,\chi)$.
\end{Corollary}
\begin{Remark}
 This was obtained by Y\i ld\i r\i m
\cite[Remark of Theorem 1]{Yi}.
However he seems to assume GRH from the context.
We stress that the above result holds
unconditionally.
\end{Remark}
We go back to the proof of Theorems \ref{Thm8}--\ref{Thm7}.
Below we assume one of the following conditions:
\begin{itemize}
 \item[(a)] $\kappa=0$ and $q\geq 216$.
 \item[(b)] $\kappa=1$ and $q\geq 23$.
\end{itemize}
We temporarily fix $\chi$ and a zero $\rho_0=\frac{1}{2}+i\gamma_0$
($\gamma_0\in\R$) of $L(s,\chi)$.
By (\ref{eq:Hadamard4}) we have
\[
 \Rep\frac{L'}{L}(s,\chi)
=\mult(\rho_0,\chi)\frac{\sigma-\frac{1}{2}}{|s-\rho_0|^2}
-\frac{1}{2}\log\frac{q}{\pi}
-\frac{1}{2}\Rep\frac{\Gamma'}{\Gamma}
\left(\frac{\rho_0+\kappa}{2}\right)+o(1)
\]
as $s\to\rho_0$,
where $\mult(\rho_0,\chi)$ is the multiplicity of the zero of $L(s,\chi)$
at $s=\rho_0$.
In the same manner as the proof of Lemma \ref{Lem:center},
the above assumption (a) or (b) implies that the constant term
$-\frac{1}{2}\log\frac{q}{\pi}
-\frac{1}{2}\Rep\frac{\Gamma'}{\Gamma}(\frac{\rho_0+\kappa}{2})$
is negative.
Thus there exists $\varepsilon=\varepsilon(\chi,\rho_0)>0$ such that
such that $\Rep(L'/L)(s,\chi)<0$ holds on
$\{s=\sigma+it:|s-\rho_0|=\varepsilon,~\sigma\leq 1/2\}$.
In the proof of Theorem \ref{Thm4}
we know that there exists $\delta>0$ satisfying
$\Rep(L'/L)(s,\chi)<0$ on $\{s=\sigma+it:|s|=\delta,~\sigma\leq 0\}$
when $\kappa=0$.
By the above discussion there exists a rectangle $\mathcal{R}$ with
vertices $\pm iT$ and $\frac{1}{2}\pm iT$ having small left semicircles
at zeros of $L(s,\chi)$
on $\Rep(s)=0$ and $\Rep(s)=1/2$ such that
$\Rep(L'/L)(s,\chi)$ is negative on the vertical sides of $\mathcal{R}$.
We apply the argument principle to $(L'/L)(s,\chi)$ on $\mathcal{R}$.
In consequence we obtain
\begin{equation}\label{eq:argprinciple}
 \frac{1}{2\pi}\Delta_{\mathcal{R}}\arg\frac{L'}{L}(s,\chi)
=N_1^{-}(T,\chi)-N^{-}(T,\chi)
-\begin{cases}
   1 & \text{if $\kappa=0$},\\
   0 & \text{if $\kappa=1$},
 \end{cases}
\end{equation}
where $\Delta_{\mathcal{R}}$ denotes the continuous variation
around the contour $\mathcal{R}$ anticlockwise.
Here we used the fact that $s=0$ is a trivial zero of $L(s,\chi)$
if $\kappa=0$.
Based on (\ref{eq:argprinciple}),
we show Theorems \ref{Thm8}--\ref{Thm7}.
Firstly we prove Theorem \ref{Thm7}.
\begin{proof}[Proof of Theorem \ref{Thm7}]
We continue to assume (a) or (b).
Since $\Rep(L'/L)(s,\chi)<0$ on the vertical sides of $\mathcal{R}$,
the continuous variation of $\arg(L'/L)(s,\chi)$ along each vertical side
is $O(1)$.
Next we investigate the horizontal sides.
We have
\begin{equation}\label{eq:arglogder}
 \left.\arg\frac{L'}{L}(s,\chi)\right|_{s=\frac{1}{2}+iT}^{s=iT}
=\left.\arg L'(s,\chi)\right|_{s=\frac{1}{2}+iT}^{s=iT}
-\left.\arg L(s,\chi)\right|_{s=\frac{1}{2}+iT}^{s=iT}.
\end{equation}
The continuous variation of $\arg L'(s,\chi)$ from $s=\frac{1}{2}+iT$ to $s=iT$
equals that of $\arg G(s,\chi)$, where the branch of $\arg G(s,\chi)$
is determined in the same manner as in $\S 3$.
Combining this with Proposition \ref{Prop:argG},
we see that the variation of $\arg L'(s,\chi)$ on (\ref{eq:arglogder})
is $O(m^{1/2}\log(qT))$.
On the other hand it is well-known that the last term on (\ref{eq:arglogder})
is $O(\log(qT))$: see \cite[Lemma 12.8]{MV} for example.
In summary we see that (\ref{eq:arglogder}) is $O(m^{1/2}\log(qT))$.
In the same manner the variation of $\arg(L'/L)(s,\chi)$
from $s=-iT$ to $s=\frac{1}{2}-iT$ is $O(m^{1/2}\log(qT))$.
We conclude that
the left-hand side of (\ref{eq:argprinciple}) is $O(m^{1/2}\log(qT))$
as desired.

We consider the case neither (a) nor (b) are satisfied.
Since the number of such characters $\chi$ is finite,
Theorem \ref{Thm7} has already been established in \cite{GS}
(see Remark of Theorem \ref{Thm7}).
We can also show this by modifying the above discussion slightly,
whose details are omitted.
\end{proof}
The following proposition is a key point to show Theorems \ref{Thm8}
and \ref{Thm9}:
\begin{Proposition}\label{Prop:Sp1}
 Let $\chi$ be a fixed primitive Dirichlet character satisfying
$\kappa=0$ and $q\geq 216$, or $\kappa=1$ and $q\geq 23$.
Then at least one of the following assertions holds:
\begin{enumerate}
 \item There exists $T_0=T_0(\chi)>0$ such that
$N^{-}(T,\chi)>T/2$ for any $T\geq T_0$.
 \item There exists a sequence $\{T_j\}_{j=1}^{\infty}$
such that $T_j\to\infty$ as $j\to\infty$
and
\[
 N_1^{-}(T_j,\chi)=N^{-}(T_j,\chi)+
\begin{cases}
1 & \text{if $\kappa=0$,}\\
0 & \text{if $\kappa=1$}
\end{cases}
\]
holds for any $j\in\Z_{\geq 1}$.
\end{enumerate}
\end{Proposition}
\begin{proof}
 First of all we suppose that there exists
a sequence $\{T_j\}_{j=1}^{\infty}$ such that
$T_j\to\infty$ as $j\to\infty$ and
both $\Rep(L'/L)(\sigma+iT_j,\chi)$
and $\Rep(L'/L)(\sigma-iT_j,\chi)$
are negative for any $j$ and $\sigma\in[0,1/2]$.
Then for any $j$,
$\Rep(L'/L)(s,\chi)$ is negative on $\mathcal{R}$ with $T=T_j$.
This implies that the left-hand side of (\ref{eq:argprinciple})
is $0$ when $T=T_j$.
In this case the assertion (2) in Proposition \ref{Prop:Sp1}
holds.

Next we suppose that $\{T_j\}_{j=1}^{\infty}$ with the above property
does not exist.
Then for any sufficiently large $t$ there exists $\sigma\in[0,1/2]$
such that $\Rep(L'/L)(\sigma+it,\chi)$ or $\Rep(L'/L)(\sigma-it,\chi)$
is nonnegative.
By Stirling's formula the first two terms on the right-hand side
of (\ref{eq:Hadamard4}) are negative for $s=\sigma+it$ or $s=\sigma-it$.
Thus $J(\sigma+it,\chi)$ or $J(\sigma-it,\chi)$
has to be negative.
This implies that
there exists a zero $\rho=\beta+i\gamma$
of $L(s,\chi)$ with $\beta<1/2$ satisfying
\[
 (\beta-\tfrac{1}{2})^2>(\sigma-\tfrac{1}{2})^2+(t-\gamma)^2
\text{ or }
(\beta-\tfrac{1}{2})^2>(\sigma-\tfrac{1}{2})^2+(t+\gamma)^2.
\]
This yields $|t-\gamma|<1/2$ or $|t+\gamma|<1/2$.
We take $t$ as a sufficiently large integer $n$.
Then we see that there exists at least one zero $\rho=\beta+i\gamma$
of $L(s,\chi)$ with $\beta<1/2$ and $n-\frac{1}{2}<|\gamma|<n+\frac{1}{2}$.
In summary we obtain $N^{-}(T,\chi)\geq T+O_{\chi}(1)$.
This implies the assertion (1) in Proposition \ref{Prop:Sp1}.
\end{proof}
\begin{proof}[Proof of Theorems \ref{Thm8} and \ref{Thm9}]
 As was mentioned in Remark of Theorems \ref{Thm8} and \ref{Thm9},
Y\i ld\i r\i m \cite{Yi} has already established
the implications (i)$\Longrightarrow$(ii).
This can be also checked by Theorem \ref{Thm1} and (\ref{eq:argprinciple}).

We suppose (ii). Then we see from
the assumption (ii) and Theorem \ref{Thm7}
that $N^{-}(T,\chi)=O_{\chi}(\log T)$.
This implies that the assertion (1) in Proposition \ref{Prop:Sp1}
cannot be satisfied.
Thus the assertion (2) in Proposition \ref{Prop:Sp1} holds.
Using the assumption (ii) again, we see  $N^{-}(T_j,\chi)=0$
for any $j$,
which is nothing but (i).
\end{proof}
Finally we mention the case when $\chi$ is quadratic.
We give a detailed information about the zero of $L'(s,\chi)$
in Theorem \ref{Thm8},
which is stated in the last sentence of the proof of Theorem 1 in
\cite{Yi}.
\begin{Proposition}\label{Prop:quadratic3}
 Let $\chi$ be an even quadratic primitive Dirichlet character
with $q\geq 216$.
We assume GRH for $L(s,\chi)$.
Then the zero of $L'(s,\chi)$ in $0<\Rep(s)<1/2$,
which is mentioned in Theorem \ref{Thm8}, is real.
\end{Proposition}
\begin{proof}
 If $\rho'$ is a non-real zero of $L'(s,\chi)$,
then $\overline{\rho'}$ is a different zero of $L'(s,\chi)$.
Combining this with Theorem \ref{Thm8}, we reach the result.
\end{proof}
 

\begin{thebibliography}{99}
\bibitem[Be]{Be}
B. C. Berndt,
The number of zeros for $\zeta^{(k)}(s)$,
J. London Math. Soc. (2) {\bf 2} (1970) 577--580.
 \bibitem[GS]{GS}
 R. Garunk\v{s}tis and R. \v{S}im\.{e}nas,
 On the Speiser equivalent for the Riemann hypothesis,
 Eur. J. Math. {\bf 1} (2015) 337--350.
 \bibitem[LM]{LM}
 N. Levinson and H. L. Montgomery,
 Zeros of the derivative of the Riemann zeta-function,
 Acta Math. {\bf 133} (1974) 49--65.
 \bibitem[MV]{MV}
 H. L. Montgomery and R. C. Vaughan,
 Multiplicative number theory I,
 Cambridge studies in advanced mathematics {\bf 97},
Cambridge University Press (2007).
\bibitem[Sp]{Sp}
A. Speiser, Geometrisches zur Riemannschen Zetafunktion,
Math. Ann. {\bf 110} (1935) 514--521.
\bibitem[Spi]{Spi}
R.Spira, Another zero-free region for $\zeta^{(k)}(s)$,
Proc. Amer. Math. Soc. {\bf 26} (1970) 246--247.
\bibitem[Ti]{Ti}
E. C. Titchmarsh,
The theory of the functions,
Oxford University Press (1939).
 \bibitem[Yi]{Yi}
C. Y. Y\i ld\i r\i m,
Zeros of derivatives of Dirichlet $L$-functions,
Turkish J. Math. {\bf 20} (1996) 521--534.
\end{thebibliography}
\end{document}